\documentclass{amsart}
\usepackage{amsfonts}
\usepackage{amssymb}
\usepackage{times}
\usepackage{xcolor}

\newtheorem{pr}{Proposition}

\newtheorem{lemma}{Lemma}
\newtheorem{de}{Definition}
\newtheorem{teo}{Theorem}

\newtheorem{remark}{Remark}
\newfont{\hueca}{msbm10}

\begin{document}
\title[Leibniz algebras associated with representations of euclidean Lie algebra]
{Leibniz algebras associated with representations of euclidean Lie
algebra}

\author{J.Q. Adashev, B.A. Omirov, S. Uguz}

\address{[J.Q. Adashev ] Institute of Mathematics. National
University of Uzbekistan, Dormon yoli str. 29, 100125, Tashkent
(Uzbekistan)}
\email{adashevjq@mail.ru}

\address{Selman Uguz, Department of Mathematics,
Arts and Sciences Faculty Harran University, 63120 \c{S}anliurfa,
Turkey} \email{{\tt selmanuguz@gmail.com}}

\address{[B.A. Omirov ] National University of Uzbekistan, University str. 4, 100174, Tashkent
(Uzbekistan)}
\email{omirovb@mail.ru}

\maketitle
\begin{abstract}
In the present paper we describe Leibniz algebras with three-dimensional Euclidean Lie algebra
$\mathfrak{e}(2)$ as its liezation. Moreover, it is assumed that the ideal generated by the squares of elements of an algebra (denoted by $I$) as a right $\mathfrak{e}(2)$-module is associated to representations of $\mathfrak{e}(2)$ in $\mathfrak{sl}_2({\mathbb{C}})\oplus \mathfrak{sl}_2({\mathbb{C}}), \mathfrak{sl}_3({\mathbb{C}})$
and $\mathfrak{sp}_4(\mathbb{C})$. Furthermore, we present the classification of Leibniz algebras with general Euclidean Lie algebra ${\mathfrak{e(n)}}$ as its liezation $I$ being an $(n+1)$-dimensional right ${\mathfrak{e(n)}}$-module defined by transformations of matrix realization of $\mathfrak{e(n)}.$ Finally, we extend the notion of a Fock module over Heisenberg Lie algebra to the case of Diamond Lie algebra $\mathfrak{D}_k$ and describe the structure of Leibniz algebras with corresponding Lie algebra $\mathfrak{D}_k$ and with the ideal $I$ considered as a Fock $\mathfrak{D}_k$-module.

\end{abstract}

\medskip \textbf{AMS Subject Classifications (2010):
17A32, 17B10, 17B30.}

\textbf{Key words:} Leibniz algebra, Euclidean Lie algebra, Diamond Lie algebra, representation of Euclidean Lie algebra, Fock module.

\section{Introduction}

Leibniz algebras are non skew-symmetric generalization of Lie algebras, in the sense that, adding antisymmetry to Leibniz bracket leads to coincidence of the fundamental identity (Leibniz  identity) with Jacobi identity. Therefore, Lie algebra is a particular case of Leibniz algebra. Leibniz algebras were introduced by J.-L. Loday \cite{Loday} in 1993 and since then the study of Leibniz algebras has been carried
on intensively. Investigation of Leibniz algebras shows that many classical results from theory of Lie algebras are extended to Leibniz algebras case (see \cite{Alb}, \cite{Balavoine}, \cite{Bar1}, \cite{Bar2}, \cite{Nilrad1}, \cite{Nilrad2}, \cite{Gorbat}, \cite{Omir} and reference therein).

For a Leibniz algebra $L$ we consider the natural homomorphism $\varphi$ into the quotient Lie algebra $\overline{L}=L/I$, which is called its {\it corresponding Lie algebra} to Leibniz algebra $L$ (in some papers it is called a {\it liezation} of $L$). The map $I \times \overline{L} \to I$, $(i,\overline{x}) \mapsto [i,x]$
endows $I$ with a structure of a right $\overline{L}$-module (it is well-defined due to $I$ being in a right annihilator).

Denote by $Q(L) = \overline{L} \oplus I,$ then the operation $(-,-)$ defines a Leibniz algebra structure on $Q(L),$ where $$(\overline{x},\overline{y}) = \overline{[x,y]}, \quad (\overline{x},i) = 0, \quad (i, \overline{x}) = [i,x], \quad (i,j) = 0, \qquad x, y \in L, \ i,j \in I.$$

Therefore, for a given Lie algebra $G$ and a right $G$-module $M,$ we can construct a Leibniz algebra as described above.

One of the approaches related to this construction is the description of Leibniz algebras with corresponding Lie algebra being a given Lie algebra. In papers \cite{Filiform}, \cite{Heisenberg} some Leibniz algebras with their corresponding Lie algebras being filiform and Heisenberg $H_n$ Lie algebras, respectively,  are described. In particular, the classification theorems for Leibniz algebras whose corresponding Lie algebras are Heisenberg in one case and naturally graded filiform algebras in another with the ideal $I$ being isomorphic to Fock module over liezation are obtained in \cite{Filiform}.

In this paper we focus our attention to Leibniz algebras constructed by Euclidean Lie algebra $\mathfrak{e(n)}$ and some of its modules. In the case $n=2$ we use modules considered in the paper \cite{dougnew}, while for Euclidean Lie algebra $\mathfrak{e(n)}$ with $n\geq 3$ we use its modules that arise from matrix realization of $\mathfrak{e(n)}.$ In addition, we clarify the structure of Leibniz algebras $Q(\mathfrak{D}_k) = \overline{\mathfrak{D}_k} \oplus I,$ where $I$ is a Fock module over Diamond Lie algebra $\mathfrak{D}_k.$ For detailed information on Diamond Lie algebra $\mathfrak{D}_k$ and its properties we refer readers to the papers \cite{Avit}, \cite{Cas}, \cite{Ludwig}.

Throughout the paper (if it is not mentioned ) we consider the base field to be $\mathbb{C}$ and in the multiplication table of an algebra omitted products are assumed to be zero.

\section{Preliminaries}

In this section we give necessary definitions and preliminary results.

\begin{de} \cite{Loday} An algebra $(L,[-,-])$ over a field  $\mathbb{F}$   is called a Leibniz algebra if for any $x,y,z\in L$ the so-called Leibniz identity
\[ \big[[x,y],z\big]=\big[[x,z],y\big]+\big[x,[y,z]\big] \] holds.
\end{de}

Let ${L}$ be a Leibniz algebra. The ideal $I$ generated by $\{[x,x]: x\in {L}\}$  plays an important role in the theory since it determines the (possible) non-Lie character of ${L}$. From the Leibniz identity, this ideal satisfies $[L,I]=0.$

\subsection{Euclidean Lie algebra and its matrix realization.}

The group $E(n)$ of Euclidean motions in the $\mathbb{R}^n$ is the noncompact semidirect product group $\mathbb{R}^n\rtimes SO(n)$. The complexification of its Lie algebra $\mathfrak{e}(n)$ admits a basis $\{E_{i,j},  \ H_{k} \mid i<j\}$ with non-zero commutation relations given by
$$[E_{i,j},E_{j,k}]=E_{i,k}, \quad [E_{i,j},H_{j}]=H_{i}, \quad [E_{i,j},H_{i}]=-H_{j},$$
assuming $E_{i,j}=-E_{j,i}$.

In fact, the matrix realization of Euclidean Lie algebra $\mathfrak{e(n)}$ can be implemented by the following matrix form:
$$\left(\begin{array}{cccccc}
&&&x_1\\
&A&&\vdots \\
&&&x_{n}\\
0&0&\dots&0\\
\end{array}\right),$$
where the matrix $A$ is self-conjugate matrix \cite{dougnew1}. In this realization $E_{i,j}=e_{i,j}-e_{j,i}, \ 1\leq i \neq j\leq n, \ H_{k}=e_{k,n+1}, \ 1\leq k\leq n$ with the matrix units $e_{i,j}$.

We preserve the usual notations $\{ p_+, p_-, l \}$ for the basis of Lie algebra $\mathfrak{e}(2)$, where $p_+=E_{1,3}, \ p_-=E_{2,3}$ and $l=E_{1,2}$.

\subsection{Diamond Lie algebra}

There is a well-know relation $\mathcal{D}_1/Center(\mathcal{D}_1) \cong \mathfrak{e}(2)$ between four-dimensional Diamond Lie algebra $\mathcal{D}_1$ and $\mathfrak{e}(2)$.

Let us consider a $(2k+2)$-dimensional real Diamond Lie algebra with a basis $\{X_i, Y_i, Z, T \ | \ 1\leq i \leq k\}$ and the table of multiplication:
$$[X_i,Y_i]=Z, \ [T,X_i]=-X_i, \ [T,Y_i]=Y_i, \ 1\leq i \leq k.$$

Take the basis transformation (complexification):

$$Z^\prime=-\frac{i}{2}Z, \ X_i^\prime=\frac{1}{2}(X_i+Y_i), \ Y_i^\prime=\frac{i}{2}(X_i-Y_i), \ T^\prime=-iT, \ 1\leq i \leq k,$$
and obtain complex Diamond Lie algebra $\mathcal{D}_k$ with the table of multiplication:
\begin{equation}\label{Diamond1}
[X_i,Y_i]=Z, \ [T,X_i]=Y_i, \ [T,Y_i]=-X_i, \ 1\leq i \leq k.
\end{equation}

In fact, Diamond Lie algebra is a double one-dimensional central extension of an abelian algebra, while Heisenberg Lie algebra is a one-dimensional extension of an abelian algebra.

\subsection{Fock module over Heisenberg Lie algebra.}

Recall, that a Heisenberg Lie algebra $H_k$ is defined by the following table of multiplications
$$[x_i,\frac{\delta}{\delta x_i}]=1, \  1\leq i \leq k,$$
in the basis $\{1, \ x_i, \ \frac{\delta}{\delta x_i}, \ 1\leq i \leq k\}.$

In the paper \cite{HeisenbergLie} the notion of a Fock module over
Heisenberg Lie algebra is introduced. Namely, it is
$\mathbb{C}[x_1,\dots,x_k]$ equipped with the following
$H_k$-module structure:
\begin{equation}\label{eqfock}
\begin{array}{lll}
(p(x_1,\dots,x_k),1)&\mapsto & p(x_1,\dots,x_k),\\
(p(x_1,\dots,x_k),x_i)&\mapsto & x_ip(x_1,\dots,x_k),\\
(p(x_1,\dots,x_k),\frac{\delta}{\delta x_i})&\mapsto &\frac{\delta}{\delta x_i}(p(x_1,\dots,x_k)),\\
\end{array}
\end{equation}
for any $p(x_1,\dots,x_k)\in \mathbb{C}[x_1,\dots,x_k]$ and $1 \leq i \leq k$.

\subsection{$\bf{\mathfrak{sl}_2(\mathbb{C})\oplus \mathfrak{sl}_2(\mathbb{C})}$-modules as $\mathfrak{e}(2)$-modules}

The special linear algebra $\mathfrak{sl}_2(\mathbb{C})$ is simple Lie algebra of traceless $2\times 2$ matrices with complex entries.
The semi-simple Lie algebra $\mathfrak{sl}_2(\mathbb{C})\oplus \mathfrak{sl}_2(\mathbb{C})$ is
of type $A_1 \times A_1$ and admits a Chevalley basis $\{x_i ,y_i
,h_j \mid 1\leq i \leq 2, 1\leq j \leq2\} $ defined as follows:

\begin{equation} ah_1+bh_2+cx_1 + dx_2+ c' y_1 + d' y_2= \left(%
\begin{array}{cccc}
  a & c & 0 & 0 \\
  c' & -a & 0 & 0 \\
  0 & 0 & b & d \\
  0 & 0 & d' & -b \\
\end{array}%
\right).
\end{equation}

Dougles, Repka and Joseph \cite{dougnew} construct
classification of embeddings of  $\mathfrak{e}(2)$   into
$\mathfrak{sl}_2(\mathbb{C}) \oplus \mathfrak{sl}_2(\mathbb{C})$ given in the next theorem.

\begin{teo}\label{teo1}
There are precisely two embeddings of $\mathfrak{e}(2)$   into
$\mathfrak{sl}_2(\mathbb{C}) \oplus
\mathfrak{sl}_2(\mathbb{C})$, up to an inner automorphism. They are
given by
$$\phi_1: \quad \quad p_+ \mapsto x_1, \quad p_- \mapsto  x_2, \quad l \mapsto   \frac{1}{2}(h_1-h_2),$$
$$\phi_2: \quad \quad p_+ \mapsto x_2, \quad p_- \mapsto x_1, \quad l \mapsto \frac{1}{2}(h_2-h_1).$$
\end{teo}

\begin{remark} \label{remark1} The basis elements $\{p_+,p_-, l \}$ of the algebra $\mathfrak{e}(2)$ by faithful representations are identified  with the linear transformations $\{\phi_1(p_+),\phi_1(p_-),\phi_1(l)\}$ of a linear space $V=\{X_1, X_2, X_3, X_4\}$.
We define on a space $V$ the structure $\mathfrak{e}(2)$-module by the action, which is naturally arises from transformations $\{\phi_1(p_+),\phi_1(p_-),\phi_1(l)\}$:

\begin{equation} \label{eq111}
\left\{\begin{array}{ll}
(X_1,p_+)= X_2,    & (X_3,p_-)= X_4,      \\[1mm]
(X_1,l)=\frac{1}{2}X_1,  & (X_2,l)=-\frac{1}{2}X_2,\\[1mm]
(X_3,l)=-\frac{1}{2}X_3,  & (X_4,l)= \frac{1}{2}X_4.\\[1mm]
\end{array}\right. \end{equation}
Note that the remaining products in the action are zero.
\end{remark}

Since $\phi_2=\varepsilon \circ \phi_1$, where $\varepsilon :  \{x_1, x_2, h_1-h_2\}  \rightarrow \{x_1, x_2, h_1-h_2\}$ with $\varepsilon(x_1)=x_2,  \varepsilon(x_2)=x_1,  \varepsilon(h_1-h_2)=h_2-h_1$ the constructed via representation $\phi_2$ module over $\mathfrak{e}(2)$ is equivalent to (\ref{eq111}).

\subsection{$\mathfrak{sl}_3(\mathbb{C})$-modules as $\mathfrak{e}(2)$-modules}

The special linear algebra $\mathfrak{sl}_3(\mathbb{C})$ is the Lie algebra
of traceless $3 \times 3$ matrices with complex entries. It is a
simple Lie algebra of type $A_2$. A Chevalley basis $\{x_i, y_i,
h_j :1\leq i \leq 3, 1 \leq j \leq 2 \} $ of $\mathfrak{sl}_3(\mathbb{C})$ is defined
as follows:

$$ ah_1 +bh_2 + cx_1 + dx_2+ ex_3 + c'y_1 + d'y_2 + e' y_3=\left(%
\begin{array}{ccc}
  a & c & -e \\
  c' & b-a &  d \\
  -e' & d' & -b \\
\end{array}%
\right).$$

Dougles et.al. \cite{dougnew} give a classification
of embeddings of  $\mathfrak{e}(2)$   into
$\mathfrak{sl}_3(\mathbb{C}) $ presented in the next statement.

\begin{teo}\label{teo2}
 There are precisely two embeddings of $\mathfrak{e}(2)$   into
$\mathfrak{sl}_3(\mathbb{C})$, up to an inner automorphism. They
are given as
$$\varphi_1: \quad \quad p_+ \mapsto x_1, \quad p_- \mapsto x_3, \quad l \mapsto -h_2,$$
$$\varphi_2: \quad \quad p_+ \mapsto  x_2, \quad p_- \mapsto x_3, \quad l \mapsto -h_1.$$
\end{teo}

Similar as in Remark \ref{remark1}, using the matrices $\{\varphi_i(p_+),\varphi_i(p_-),\varphi_i(l)\}, \ i=1,2$ from Theorem \ref{teo2}, we define non-isomorphic $\mathfrak{e}(2)$-module structures on a vector space $V=\{X_1, X_2, X_3\}$ as follows:
\begin{equation}  \label{eq5}    \left\{\begin{array}{ll}
(X_1,p_+)= X_2,    & (X_1,p_-)= -X_3,      \\[1mm]
(X_2,l)=-X_2,  & (X_3,l)=X_3,
\end{array}\right. \end{equation}

\begin{equation} \label{eq6} \left\{\begin{array}{ll}
(X_1,p_-)= -X_3,    & (X_2,p_+)= X_3,      \\[1mm]
(X_1,l)=-X_1,  & (X_2,l)=X_2,\\[1mm]
\end{array}\right. \end{equation}
the remaining products in the actions are zero.

\subsection{$\mathfrak{sp}_4(\mathbb{C})$-modules as $\mathfrak{e}(2)$-modules}

The symplectic algebra $\mathfrak{sp}_4(\mathbb{C})$ is the Lie algebra of
$4 \times 4$ complex matrices $X$ satisfying $J X^T J=X$, where
$J$ is a
$4\times  4$ matrix  $$J=\left(%
\begin{array}{cc}
  0 & I_2 \\
  -I_2 &0 \\
\end{array}%
\right).$$

It is a 10-dimensional simple Lie algebra of type $C_2$
(equivalent to the simple Lie algebra of type $B_2$). A Chevalley
basis $\{ x_i, y_i, h_j \mid 1 \leq i \leq 4, 1 \leq j \leq 2 \}$
of $\mathfrak{sp}_4(\mathbb{C})$ is defined as follows:
$$ ah_1 + bh_2 + cx_1 + dx_2 + ex_3 + fx_4 + c' y_1 + d' y_2 + e' y_3 + f' y_4 = \left(%
\begin{array}{cccc}
  a & c & f & -e \\
  c' & -a+b & -e &  d \\
  f' & -e' & -a & -c' \\
  -e' & d' & -c & a-b \\
\end{array}%
\right).$$

  Dougles et.al. \cite{dougnew} present classification
of embeddings of  $\mathfrak{e}(2)$   into
$\mathfrak{sp}_4(\mathbb{C})$ in the next theorem.

\begin{teo}\label{teo3}
 There are three families of embeddings of $\mathfrak{e}(2)$   into
$\mathfrak{sp}_4(\mathbb{C} $, up to inner automorphism. Two
families contain a single embedding, and one family is infinite.
They are given as
$$\psi_1:\quad \quad p_+ \mapsto x_4,\quad p_- \mapsto x_3,\quad l \mapsto \frac{1}{2}h_1-h_2,$$
$$\psi_2:\quad\quad p_+ \mapsto x_3, \quad p_-\mapsto  x_4,\quad l \mapsto -\frac{1}{2}h_1 +h_2,$$
$$\psi_{3,\beta}:\quad\quad p_+\mapsto x_2,\quad p_- \mapsto x_4,\quad l \mapsto -\frac{1}{2}h_1 + \beta x_3,$$ where
$\beta \in \mathbb{C}$ and $\psi_{3,\alpha}\sim \psi_{3,\beta} $ iff
$\alpha^2=\beta^2$.
\end{teo}

Similarly as before, we have $\psi_2=\varepsilon \circ \psi_1$, where $\varepsilon :  \{x_3, x_4, \frac{1}{2}h_1-h_2\}  \rightarrow \{x_3, x_4, \frac{1}{2}h_1-h_2\}$ with $\varepsilon(x_4)=x_3,  \varepsilon(x_3)=x_4,  \varepsilon(\frac{1}{2}h_1-h_2)=-\frac{1}{2}h_1+h_2$. Therefore, it is not necessary to consider the module constructed via representation $\psi_2$.

Analogously as in Remark \ref{remark1} applying results of Theorem \ref{teo3}, by the transformations $\psi_1(J),\psi_1(P_{+}),\psi_1(P_{-})$ and $\{\psi_3(p_+),\psi_3(p_-),\psi_3(l)\}$ we define two non-isomorphic $\mathfrak{e}(2)$-module structures on a vector space $V=\{X_1,X_2,X_3,X_4\}$ in a similar way:
\begin{equation} \label{eq7} \left\{\begin{array}{lll}
(X_1,p_+)= X_3,   & (X_1,p_-)=-X_4, & (X_2,p_{-})=-X_3, \\[1mm]
(X_1,l)=\frac{1}{2}X_1,    & (X_2,l)=-\frac{3}{2}X_2, \\[1mm]  (X_3,l)=-\frac{1}{2}X_3, & (X_4,l)= \frac{3}{2} X_4,
\end{array}\right. \end{equation}

\begin{equation} \label{eq8} \left\{\begin{array}{ll}
(X_2,p_+)= X_4,    & (X_1,p_-)= X_3,      \\[1mm]
(X_1,l)=-\frac{1}{2}X_1-\beta X_4,  & (X_2,l)=\frac{1}{2}X_2-\beta X_3.\\[1mm]
(X_3,l)= \frac{1}{2}X_3,  & (X_4,l)=-\frac{1}{2}X_4.\\[1mm]
\end{array}\right.  \end{equation}

\section{Main results}

\subsection{Leibniz algebras associated with representation of Euclidean Lie algebra $\mathfrak{e}(2)$
considered as a subalgebra of $\mathfrak{sl}_2(\mathbb{C})\oplus
\mathfrak{sl}_2(\mathbb{C}) $ }

In this subsection we describe Leibniz algebras $L$ with $L/ I \cong \mathfrak{e}(2)$ and the ideal $I$ being a four-dimensional
$\mathfrak{e}(2)$-module defined by (\ref{eq111}). In this case an algebra $L$ have a basis $\{ l, p_+, p_-, X_1, X_2, X_3, X_4\}.$

Let us introduce denotations
\begin{equation}\label{eq99}\left\{\begin{array}{ll}
[l,p_+]=p_+ +  \sum\limits_{i=1}^4f_iX_i,    &[p_+,l]=-p_+ +  \sum\limits_{i=1}^4a_iX_i, \\[1mm]
[l,p_-]=-p_- +  \sum\limits_{i=1}^4g_iX_i, &[p_-,l]=p_-+\sum\limits_{i=1}^4b_iX_i, \\[1mm]
[l,l]= \sum\limits_{i=1}^4c_iX_i,  & [p_+,p_+]=\sum\limits_{i=1}^4d_iX_i, \\[1mm]
[p_-,p_-]=\sum\limits_{i=1}^4e_iX_i, & [p_+,p_-]=\sum\limits_{i=1}^4m_iX_i, \\[1mm]
[p_-,p_+]=\sum\limits_{i=1}^4n_iX_i. &   \\[1mm]
\end{array}\right.\end{equation}

\begin{teo} Let $L$ be a Leibniz algebra with an associated Lie algebra
$\mathfrak{e}(2)$ and the ideal $I$ being the
$\mathfrak{e}(2)$-module defined by (\ref{eq111}).
Then there exists a basis $\{ l,p_+,p_-, X_1, X_2, X_3, X_4\}$ of an algebra $L$ such
that its table of multiplication is in the following form:
\begin{equation}\label{eq10}  \left\{\begin{array}{ll}
[l,p_+]=p_+,    &[p_+,l]=-p_+, \\[1mm]
[l,p_-]=-p_-, &[p_-,l]=p_-, \\[1mm]
[X_1,p_+]=X_2, & [X_3,p_-]=X_4,\\[1mm]
[X_1,l]=\frac{1}{2}X_1, &   [X_2,l]=-\frac{1}{2}X_2, \\[1mm]
[X_3,l]=-\frac{1}{2}X_3,  & [X_4,l]= \frac{1}{2}X_4.
\end{array}\right.\end{equation}
\end{teo}

\begin{proof}
Take a change of basis elements $\{l, p_+, p_-\}$ as follows:
$$l'= l -2c_1 X_1+ 2c_2X_2 + 2c_3 X_3-2c_4 X_4, \ \  p_+'=p_+-2c_1 X_2 +
\sum\limits_{i=1}^4f_iX_i, \ \  p_-' =p_- -2c_3 X_4 -
\sum\limits_{i=1}^4g_iX_i.$$ Then by using notations (\ref{eq99}) we conclude that
$$ [l',l']=0, \ \ [l',p_+']=p_+', \ \ [l',p_-']=-p_-'.
$$

Considering Leibniz identity for triples mentioned below we obtain
$$\left\{\begin{array}{lll}
(l',p_+',p_-') & \Rightarrow & m_1=-n_1, m_3=-n_3, m_2=-n_2-g_1, m_4=-n_4-f_3,\\[1mm]
(l',p_+',l') & \Rightarrow & \frac{3}{2}f_1=-a_1,\frac{1}{2}f_2=c_1-a_2, \frac{1}{2}f_3=-a_3, \frac{3}{2}f_4=-a_4, \\[1mm]
(l',p_-',l') & \Rightarrow & b_1=-\frac{1}{2}g_1, b_2=-\frac{3}{2}g_2, b_3=-\frac{3}{2}g_3, b_4=-\frac{1}{2}g_4-c_3, \\[1mm]
(p_+',l',p_+') & \Rightarrow & d_1=d_3=d_4=0, \frac{3}{2}d_2=a_1, \\[1mm]
(p_-',l',p_-') & \Rightarrow & e_1=e_2=e_3=0,  \frac{3}{2}e_4=-b_3,  \\[1mm]
(p_+',l',p_-') & \Rightarrow & m_1=m_2=m_3=0, m_4=2a_3,  \\[1mm]
(p_-',l',p_+') & \Rightarrow & n_1=n_3=n_4=0, n_2=-2b_1, \\[1mm]
\end{array}\right.$$

From these we obtain the table of multiplication (\ref{eq10}).
\end{proof}

\subsection{Leibniz algebras associated with representation of Euclidean Lie algebra $\mathfrak{e}(2)$ considered as a subalgebra of $\mathfrak{sl}_3(\mathbb{C})$ }

Let $L$ be a Leibniz algebra such that $L/ I \cong \mathfrak{e}(2)$ and the ideal $I$ of $L$ is an $\mathfrak{e}(2)$-module defined by either (\ref{eq5}) or (\ref{eq6}).

The following proposition defines the products of basis elements.

\begin{pr} Let $L$ be a Leibniz algebra with associated Lie algebra
$\mathfrak{e}(2)$ and $I$ be an $\mathfrak{e}(2)$-module defined by (\ref{eq5}).
Then there exists a basis $\{ l, p_+, p_-, X_1, X_2, X_3\}$ of $L$ such that its table of multiplications has the following form:
$$ K_1(\alpha_1,\alpha_2): \quad \left\{\begin{array}{ll}
[l, p_+]=p_+,    & [p_+, l]=-p_+, \\[1mm]
[l, p_-]=-p_-, & [p_-, l]=p_-, \\[1mm]
[l, l]=  \alpha_1 X_1,  & [p_+, p_-]= \alpha_2 X_1, \\[1mm]
[p_-,p_+]=-\alpha_2 X_1, & [X_1,p_+]= X_2,   \\[1mm]
     [X_1,p_-]= -X_3, & [X_2,l]=-X_2,     \\[1mm]
    [X_3,l]=X_3.
\end{array}\right.
$$
\end{pr}
\begin{proof}

Let us set the products of basis elements $\{p_+,p_-,l\}$ similar as in (\ref{eq99})but without $X_4$ participation.

Taking the change of basis elements $\{l, p_+, p_-\}$ as follows
$$l'= l -c_1X_1 +c_2 X_2 -c_3 X_3, \ \  p_+'=p_+ -c_1X_2+ \sum\limits_{i=1}^3f_iX_i, \ \  p_-' =p_- - c_1X_3 - \sum\limits_{i=1}^3g_iX_i,$$
we can assume
$$ [l, l]=c_1X_1, \ \ [l, p_+]=p_+, \ \ [l, p_-]=-p_-.
$$

Considering Leibniz identities for different triples we deduce
$$\left\{\begin{array}{lll}
(l, p_+, p_-) & \Rightarrow & m_1=-n_1, \ m_2=-n_2 +g_1, \ m_3=-n_3 + f_1,   \\[1mm]
(l, p_+, l) & \Rightarrow & a_1=-f_1, a_2=0, a_3=-2f_3, \\[1mm]
(l, p_-, l) & \Rightarrow & b_1=-g_1, \ b_2=-2g_2, \ b_3=0, \\[1mm]
(p_+, l, p_+) & \Rightarrow & d_1=d_3=0, \ d_2=a_1, \\[1mm]
(p_-, l, p_-) & \Rightarrow & e_1=e_2=0, e_3=b_1,  \\[1mm]
(p_+, l, p_-) & \Rightarrow & m_2=0, m_3=-a_1,  \\[1mm]
(p_-, l, p_+) & \Rightarrow & n_2=-b_1, n_3=0. \\[1mm]
\end{array}\right.$$

Putting $\alpha_1:=c_1$ and $\alpha_2:=m_1$ we get the family of algebras $K_1(\alpha_1,\alpha_2).$
\end{proof}

In the next theorem we identify the representatives (up to isomorphism) of the family of algebras $K_1(\alpha_1,\alpha_2).$

\begin{teo}\label{theorem5} An arbitrary algebra of the family $K_1(\alpha_1,\alpha_2)$ is isomorphic to
one of the following pairwise non-isomorphic algebras:
$$K_1(1,1), \quad K_1(1,0) \quad K_1(0,1) \quad K_1(0,0).$$
\end{teo}

\begin{proof} In order to achieve our goal we consider isomorphism (basis transformation) inside the family
$K_1(\alpha_1,\alpha_2).$ Since $\{l, p_+, p_-, X_1\}$ are
generators of the algebra, we take the general
transformation of these basis elements:
$$\begin{array}{ll}
l'=A_1l+A_2p_+ +A_3p_- + \sum\limits_{i=1}^3A_{i+3}X_i, & p_+'=B_1 l+ B_2 p_+ +B_3 p_- + \sum\limits_{i=1}^3B_{i+3}X_i,\\[3mm]
p_-'=C_1 l+ C_2 p_+ +C_3 p_- + \sum\limits_{i=1}^3C_{i+3}X_i, & X_1'=\sum\limits_{i=1}^3D_{i}X_i.\\[3mm]
\end{array}$$

Then the rest of the basis elements are obtained from the products $X_2'=[X_1',p_+']$ and $X_3'=-[X_1',p_-'],$ that is,
$$X_2'=(D_1 B_2 - D_2 B_1) X_2 + (D_3 B_1 - D_1 B_3)X_3, \quad X_3'=(D_2 C_1 - D_1 C_2) X_2 + (D_1C_3 -D_3C_1) X_3.$$

Now we apply the following procedure:

\begin{itemize}

\item

in the first step we obtain all the products $[ *' , * ']$ by substituting the above basis transformation and applying the products $[* , * ]$;

\item

in the second step in the expression of the products $[ *' , * ']$ of the algebra $K_1(\alpha_1',\alpha_2')$ in the basis
$\{ l', p_+', p_-', X_1', X_2', X_3'\}$ we substitute the above basis transformation;

\item

in the third step comparing two expressions obtained in the previous steps we derive the expressions for $\alpha_1', \ \alpha_2'$ in terms of parameters $\alpha_1, \ \alpha_2, A_i, B_i, C_i, D_i.$

\end{itemize}

Applying the procedure gives the following expressions:
$$\alpha_1'=\frac{\alpha_1}{D_1}, \quad \quad \alpha_2'=\frac{B_2C_3\alpha_2}{D_1} \ \mbox{with} \ B_2C_3D_1\neq 0.$$

{\bf Case 1.}  Let $\alpha_2\neq 0.$ By putting $D_1=B_2C_3\alpha_2$ we get $\alpha_2'=1$
 and $\alpha_1'=\frac{\alpha_1}{B_2C_3\alpha_2}.$

If $\alpha_1=0$ then we obtain the algebra $K_1(0,1)$.

If $\alpha_1\neq 0$ then taking $B_2=\frac{\alpha_1}{C_3\alpha_2}$ we get $K_1(1,1)$.

{\bf Case 2.} Let $\alpha_2=0.$ Then $\alpha_2'=0$.

If $\alpha_1=0$ then we obtain the algebra $K_1(0,0)$.

If $\alpha_1\neq0$ then taking $D_1= \alpha_1 $ we have $K_1(1,0)$.
\end{proof}

In a similar way we derive the corresponding results in the case when the ideal $I$ is an $\mathfrak{e}(2)$-module defined by (\ref{eq6}).

\begin{pr} Let $L$ be a Leibniz algebra with associated Lie algebra
$\mathfrak{e}(2)$ and $I$ be an $\mathfrak{e}(2)$-module defined by (\ref{eq6}).
Then there exists a basis $\{ l, p_+, p_-, X_1, X_2, X_3\}$ of $L$ such that its table of multiplications has the following form:
\begin{equation}\label{eq11} K_2(\alpha_1,\alpha_2)=\left\{\begin{array}{ll}
[l,p_+]=p_+,    &[p_+,l]=-p_+, \\[1mm]
[l,p_-]=-p_-, &[p_-,l]=p_-, \\[1mm]
[l,l]=  \alpha_1 X_3,  &
  [p_+,p_-]= \alpha_2 X_3, \\[1mm]
[p_-,p_+]=-\alpha_2 X_3, &
[X_1,p_-]= -X_3,   \\[1mm]
    [X_1,l]= -X_1, & [X_2,p_+]=X_3,     \\[1mm]
    [X_2,l]=X_2.
\end{array}\right.\end{equation}
\end{pr}

\begin{teo}\label{theorem5} An arbitrary Leibniz algebra of the family of algebras $K_2(\alpha_1,\alpha_2)$ is isomorphic to one of the following pairwise non-isomorphic algebras:
$$K_2(1,1), \quad K_2(1,0) \quad K_2(0,1) \quad K_2(0,0).$$
\end{teo}

\begin{remark}
Let $H_1,H_2,E_1,E_2,E_{12},F_1,F_2$ and $F_{12}$ be Chevalley basis of $\mathfrak{sl}_3(\mathbb{C})$ defined by
$$aH_1+bH_2+cE_1+dE_2+eE_{12}+fF_1+gF_2+hF_{12}=\left(\begin{matrix}a&c&e&\\f&b-a&d\\h&g&-b\end{matrix}\right).$$

Dougles and Premat \cite{doug} construct indecomposable finite-dimensional representations of $\mathfrak{e}(2)$ by restricting those of $\mathfrak{sl}_3(\mathbb{C})$ to one embedding of $\mathfrak{e}(2)$ in $\mathfrak{sl}_3(\mathbb{C})$ given the next lemma.
\end{remark}

\begin{lemma}\label{lemma1} A map $\varphi:\mathfrak{e}(2)\rightarrow \mathfrak{sl}_3(\mathbb{C})$
defined on the generators of $\mathfrak{e}(2)$   by
$$\varphi(p_+)=E_1, \quad \varphi(p_-)=F_{2}, \quad \varphi(l)=
H_1+H_2,$$ is a   Lie algebra embedding.
\end{lemma}

We construct a module $V$ by action $V \times  \mathfrak{e}(2)  \rightarrow V$ defined by linear transformations with the matrices $\{\varphi(p_+),\varphi(p_-),\varphi(l)\}$ on the linear space $V=\{X_1, X_2, X_3\}$. Then we obtain
\begin{equation} \label{9}\left\{\begin{array}{ll}
(X_1,p_+)= X_2,    & (X_3,p_-)= X_2,      \\[1mm]
(X_1,l)=X_1,  & (X_3,l)=-X_3.
\end{array}\right.\end{equation}

\begin{remark}
If we consider the following change of basis elements
\begin{eqnarray*}
l'=-l, \ p_+'=p_-, \ p_-'=p_+, \ \
X_1'= X_1, \ X_2'= -X_3, \ X_3'= -X_2,
\end{eqnarray*} then the action \eqref{9} transforms to the action (\ref{eq6}). Hence the description of Leibniz algebras constructed by $\mathfrak{e}(2)$-module \eqref{9} are already obtained in  Theorem \ref{theorem5}. \end{remark}

\subsection{Leibniz algebras associated with representation of Euclidean Lie algebra $\mathfrak{e}(2)$ considered as a subalgebra of $\mathfrak{sp}_4(\mathbb{C})$ }

In this section we describe the Leibniz algebras $L$ such that $L/ I \cong \mathfrak{e}(2)$ and the ideal $I$ is a right $\mathfrak{e}(2)$-module with action either \eqref{eq7} or \eqref{eq8}.

\begin{teo} \label{thm7} Let $L$ be a Leibniz algebra with associated Lie algebra $\mathfrak{e}(2)$ and the ideal $I$ be a right $\mathfrak{e}(2)$-module defined by (\ref{eq7}). Then there exists a basis $\{ l,p_+,p_-, X_1, X_2,X_3,X_4\}$ of $L$ such that its table of multiplications has the following form:
$$\left\{\begin{array}{lll}
[l,p_+]=p_+,    &[p_+,l]=-p_+, \\[1mm]
[l,p_-]=-p_-, &[p_-,l]=p_-, \\[1mm]
[X_1,p_+]= X_3,   &  [X_1,p_-]=-X_4, \\[1mm] [X_2,p_{-}]=-X_3, &
[X_1,l]=\frac{1}{2}X_1,    \\[1mm] [X_2,l]=-\frac{3}{2}X_2,  & [X_3,l]=  -\frac{1}{2} X_3 \\[1mm] [X_4,l]=  \frac{3}{2} X_4.
\end{array}\right.$$
\end{teo}
\begin{proof} Here for the products of the elements $\{p_+,p_-,l\}$ we  use notations of \eqref{eq99}.

Consider the change of the basis elements $l, p_+,p_-$ in the following way
$$\begin{array}{lll}
l'= l -2c_1X_1 +\frac{2}{3}c_2 X_2 +2c_3 X_3-\frac{2}{3}c_4X_4, \\[1mm]
p_+'=p_+ -2c_1X_3+ \sum\limits_{i=1}^4f_iX_i, \\[1mm]
p_-'=p_- - 2c_1X_4 +\frac{2}{3}c_2X_3 - \sum\limits_{i=1}^4g_iX_i.
\end{array}$$

Then we can assume
$$ [l,l]=0, \ \ [l,p_+]=p_+, \ \ [l,p_-]=-p_-.$$

From the following Leibniz identities we derive
$$\left\{\begin{array}{lll}
(l',p_+',p_-'), & \Rightarrow & m_1=-n_1, \ m_2=-n_2, \ m_3=-n_3 +g_1 + f_2, \ m_4=-n_4  + f_1,   \\[1mm]
(l',p_+',l'), & \Rightarrow & a_1=-\frac{3}{2}f_1, a_2=\frac{1}{2}f_2, a_3=- \frac{1}{2}f_3 +c_1, a_4=-\frac{5}{2}f_4, \\[1mm]
(l',p_-',l'), & \Rightarrow & b_1=-\frac{1}{2}g_1, \ b_2=- \frac{5}{2}g_2, \ b_3= -\frac{3}{2}g_3 + c_2, b_4= \frac{1}{2}g_4 + c_1, \\[1mm]
(p_+',l',p_+'), & \Rightarrow & d_1=d_2=d_4=0, \ d_3=\frac{2}{3}a_1, \\[1mm]
(p_-',l',p_-'), & \Rightarrow & e_1=e_2=0, e_3=\frac{2}{5}b_1,   e_4=2b_1, \\[1mm]
(p_+',l',p_-'), & \Rightarrow & m_1=m_2=0, m_3=2a_2,   m_4=\frac{2}{3}a_1, \\[1mm]
(p_-',l',p_+'), & \Rightarrow & n_1=n_2=n_4=0, n_3=-2b_1. \\[1mm]
\end{array}\right.$$

Hence we get
$$[l',p_+']=p_+', \quad [p_+',l']=-p_+', \quad [l',p_-']=-p_-', \quad [p_-',l']=p_-'.$$
\end{proof}

\begin{teo} \label{thm12} Let $L$ be a Leibniz algebra with associated Lie algebra
$\mathfrak{e}(2)$ and the ideal $I$ be a right $\mathfrak{e}(2)$-module defined by (\ref{eq8}).
Then there exists a basis $\{ l,p_+,p_-, X_1, X_2, X_3, X_4\}$ of $L$ such that its table of multiplications has the following form:
$$\left\{\begin{array}{lll}
[l,p_+]=p_+,    &[p_+,l]=-p_+, \\[1mm]
[l,p_-]=-p_-, &[p_-,l]=p_-, \\[1mm]
[X_2,p_+]= X_4,     & [X_1,p_-]= X_3, \\[1mm]
[X_1,l]=-\frac{1}{2}X_1-\beta X_4,  & [X_2,l]=\frac{1}{2}X_2-\beta X_3, \\[1mm] [X_3,l]=\frac{1}{2}  X_3, & [X_4,l]= -\frac{1}{2}X_4.
\end{array}\right.$$
\end{teo}
\begin{proof} Similarly as in the proof of Theorem \ref{thm7} considering the change of elements $\{l, p_+, p_-\}$ as follows
$$l'= l +2c_1X_1 -4c_1 \beta X_4 -2c_2 X_2 -4c_2\beta  X_3 -2c_3X_3 + 2c_4X_4,$$
$$p_+'=p_+ -2c_2X_4+ \sum\limits_{i=1}^3f_iX_i, \quad p_-' =p_- -2c_1X_3 - \sum\limits_{i=1}^3g_iX_i,$$
yields
$$[l,l]=0, \quad [l,p_+]=p_+, \quad [l,p_-]=-p_-.
$$

The proof of the theorem completes the following verifications of Leibniz identities
$$\left\{\begin{array}{lll}
(l',p_+',p_-'), & \Rightarrow & m_1=-n_1, \ m_2=-n_2, \ m_3=-n_3 - f_1, m_4=-n_4 +g_2,  \\[1mm]
(l',p_+',l'), & \Rightarrow & a_1=-2f_1, a_2=-\frac{3}{2}f_2, a_3=(-\frac{3}{2}+\beta)f_3, a_4= -\frac{1}{2}f_4+\beta f_1 +c_2, \\[1mm]
(l',p_-',l'), & \Rightarrow & b_1=-\frac{3}{2}g_1, \ b_2=-\frac{1}{2}g_2, \ b_3=-\frac{1}{2}g_3-c_1-\beta g_2, b_4=-\frac{3}{2}g_4 -\beta g_1,\\[1mm]
(p_+',l',p_+'), & \Rightarrow & d_1=d_2=d_3=0, \ d_4=\frac{2}{3} a_2, \\[1mm]
(p_-',l',p_-'), & \Rightarrow & e_1=e_2=e_4=0, e_3=-\frac{2}{3}b_1,  \\[1mm]
(p_+',l',p_-'), & \Rightarrow & m_1=m_2=m_4=0, m_3=-\frac{1}{2}a_1,  \\[1mm]
(p_-',l',p_+'), & \Rightarrow & n_1=n_2=n_3=0, n_4=-2b_2. \\[1mm]
\end{array}\right.$$
\end{proof}

\subsection{Leibniz algebras associated with representation of Euclidean Lie algebra $\mathfrak{e(n)}$ realized by its matrix realization.} \

In order to distinguish the quotient Lie algebra $L/I$ and its preimage under natural homomorphism, for the quotient algebra we shall use notation with a line at the top.

In this subsection we describe Leibniz algebras $L$ such that $L/I\cong {\mathfrak{\overline{e}(n)}}$ (here the quotient algebra  ${\mathfrak{\overline{e}(n)}}$ means the algebra ${\mathfrak{{e}(n)}}$) and the ideal $I$ is an $(n+1)$-dimensional  right ${\mathfrak{\overline{e}(n)}}$-module with a basis $\{X_1, \dots, X_{n+1}\}$, which defined by transformations of matrix realization of ${\mathfrak{\overline{e}(n)}}$:

\begin{equation} \label{eq1111}
[X_i,\overline{E}_{i,j}]=X_{j},\ 1\leq i,j\leq n,\quad [X_i,\overline{H}_{i}]=X_{n+1},\ \ 1\leq i\leq n.
\end{equation}

It should be noted that ${\mathfrak{\overline{e}(n)}}\cong \mathfrak{so}_{n-1}\dot{+} \mathbb{C}^n$, where $\mathfrak{so}_{n-1}$ is orthogonal simple Lie algebra \cite{dougnew1} with the basis $\overline{E}_{i,j}, i < j$.

Thus we have an algebra $L\cong (\mathfrak{so}_{n-1}\dot{+} \mathbb{C}^n) + I$ with a basis $\{E_{i,j}, i < j, H_{k}, X_1, \dots, X_{n+1}\}$, where elements $E_{i,j}, H_{k}$ are pre-image of corresponding elements of the quotient algebra $L/I$. Due to Levi's theorem \cite{Bar1} we conclude that
$\mathfrak{so}_{n-1}$ is a subalgebra of $L$, that is, $[E_{i,j}, E_{j,k}]=E_{i,k}$.

\begin{teo} Let $L$ be a Leibniz algebra such that $L/I \cong \mathfrak{\overline{e}(n)}$ and the ideal $I$ is a right $\mathfrak{\overline{e}(n)}$-module defined by \eqref{eq1111}. Then $[\mathfrak{e(n)}, \mathfrak{e(n)}]= \mathfrak{e(n)}.$
\end{teo}
\begin{proof} We set
$$\begin{array}{ll}
[E_{i,j},H_{i}]=-H_j+\sum\limits_{t=1}^{n+1}A_{i,j,i}^tX_t,&[H_{i},E_{i,j}]=H_j+\sum\limits_{t=1}^{n+1}B_{i,i,j}^tX_t,\\[1mm]
[E_{i,j},H_{j}]=H_i+\sum\limits_{t=1}^{n+1}A_{i,j,j}^tX_t,&[H_{j},E_{i,j}]=-H_i+\sum\limits_{t=1}^{n+1}B_{j,i,j}^tX_t,\\[1mm]
[E_{i,j},H_{k}]=\sum\limits_{t=1}^{n+1}A_{i,j,k}^tX_t,\ k\notin\{i,j\},&[H_{k},E_{i,j}]=\sum\limits_{t=1}^{n+1}B_{k,i,j}^tX_t,\ k\notin\{i,j\},\\[1mm]
[H_{i},H_{j}]=\sum\limits_{t=1}^{n+1}C_{i,j}^tX_t.\\[1mm]
\end{array}$$

Taking the change
$$H_{1}^\prime=H_1+\sum\limits_{t=1}^{n+1}A_{1,2,2}^tX_t,\quad H_{j}^\prime=H_j-\sum\limits_{t=1}^{n+1}A_{1,j,1}^tX_t, \ 2\leq j\leq n,$$
we can assume that $$[E_{1,2},H_{2}]=H_1,\quad [E_{1,j},H_{1}]=-H_j, \
2\leq j\leq n,$$

For $2\leq i\neq j \leq n$ we consider the Leibniz identity
$$[E_{j,i},[E_{1,i},H_{1}]]=[[E_{j,i},E_{1,i}],H_1]-[[E_{j,i},H_1],E_{1,i}]=-H_j-A_{j,i,1}^1X_i+A_{j,i,1}^iX_1.$$

On the other hand we have
$$[E_{j,i},[E_{1,i},H_{1}]]=-[E_{j,i},H_i]=[E_{i,j},H_i].$$

Consequently,
\begin{equation} \label{e(n)eq1}
[E_{i,j},H_i]=-H_j+A_{j,i,1}^iX_1-A_{j,i,1}^1X_i, \ 2\leq i\neq j \leq n.
\end{equation}

Similarly, we obtain
\begin{equation} \label{e(n)eq2}
[E_{k,l},H_i]=-A_{k,l,1}^iX_1+A_{k,l,1}^1X_i, 2\leq i, k,l \leq n, \ i\notin \{k, l\}.
\end{equation}

Applying Equality (\ref{e(n)eq1}) and taking into account that $A_{i,j,1}^1=-A_{j,i,1}^1$ in
the Leibniz identity for the triples of elements
$$\begin{array}{lll}
\{E_{i,j},H_i,E_{i,j}\}, & 2\leq i\neq j \leq n,\\[1mm]
\{E_{i,j},H_i,E_{k,l}\}, & 1\leq i\neq j \leq n,  \ \ 2\leq k\neq l \leq n, & \{i,j\}\cap\{k,l\}=\{0\},\\[1mm]
\end{array}
$$
we derive
\begin{equation}\label{e(n)eq3}\begin{array}{lll}
[H_j,E_{i,j}]=-H_i+A_{i,j,1}^jX_1, &  2\leq i\neq j \leq n,\\[1mm]
[H_j,E_{k,l}]=0, & 1\leq j \leq n, \ \  2\leq k\neq l \leq n, \ \  j\notin \{k,l\}.\\[1mm]
\end{array}
\end{equation}

Using (\ref{e(n)eq1})-(\ref{e(n)eq2}) in the Leibniz identity for the elements $\{E_{i,j},H_k,E_{k,l}\}$ we conclude $$A_{i,j,1}^1=0,  \quad 2 \leq i \neq j \leq n \Rightarrow
[E_{k,l},H_i]=-A_{k,l,1}^iX_1, \quad 2\leq i, k,l \leq n, \ i\notin \{k, l\}.$$

Analogously, from \ref{e(n)eq2} - (\ref{e(n)eq3}) and Leibniz identity for the triples $\{E_{1,2},E_{i,2},H_i\}, \{E_{i,j}, E_{1,2}, H_{2}\}, \ 3\leq i, j  \leq n$ we derive
$$[E_{i,1},H_i]=-H_1+A_{1,2,i}^iX_2-A_{1,2,i}^2X_i,\quad [E_{i,j},H_1]=A_{i,j,1}^2X_2, \ \ 3\leq i,j \leq n.$$

The chain of equalities
$$0=[E_{i,j},[H_{1},H_{2}]]=[[E_{i,j},H_{1}],H_2]-[[E_{i,j},H_2],H_{1}]=2A_{i,j,1}^2X_2$$
imply
$$A_{i,j,1}^k=0,\ \ 1\leq k\leq n, \ 3\leq i,j \leq n \Rightarrow [E_{i,j},H_1]=0.$$

Thus, we get
$$\begin{array}{lll}
[E_{i,j},H_i]=-H_j, &  3\leq i\neq j \leq n,\\[1mm]
$$[H_j,E_{i,j}]=-H_i, & 3\leq i\neq j \leq n,\\[1mm]
$$[H_j,E_{k,l}]=0, & 1\leq j \leq n, \ \  2\leq k\neq l \leq n, \ \  j\notin \{k,l\}.\\[1mm]
\end{array}$$

For $3\leq i,j,k\leq n$ we consider the Leibniz identity
$$[E_{i,1},[E_{j,k},H_j]]=[[E_{i,1},E_{j,k}],H_j]-[[E_{i,1},H_j],E_{j,k}]=A_{i,1,j}^jX_k-A_{i,1,j}^kX_j.$$

On the other hand, we have
$$[E_{i,1},[E_{j,k},H_j]]=-[E_{i,1},H_k].$$

Therefore, we conclude
\begin{equation}\label{e(n)eq4}
[E_{i,1},H_k]=A_{i,1,j}^kX_j-A_{i,1,j}^jX_k.
\end{equation}

Since in the right side of \eqref{e(n)eq4} there is a free parameter $j$ and on the left side of \eqref{e(n)eq4} we there is not it follows that $[E_{i,1},H_k]=0.$

Considering the Leibniz identity for the following triples we deduce
$$\begin{array}{lllll}
\{E_{j,i}, H_j, E_{1,k}\}, & 3\leq i, j, k\leq n  & \Rightarrow & [H_j,E_{1,i}]=0, & 3\leq i, j\leq n, \\[1mm]
\{E_{2,j}, H_{2}, E_{2,j}\}, & 3\leq j\leq n  & \Rightarrow & [H_j,E_{2,j}]=-H_2, & 3\leq j\leq n,\\[1mm]
\{E_{i,1}, H_{i}, E_{i,j}\}, & 3\leq i\neq j \leq n & \Rightarrow & A_{1,2,i}^2=0, & 3\leq i \leq n, \\[1mm]
\{E_{2,j}, H_2, E_{1,k}\}, & 3\leq j, k\leq n & \Rightarrow & A_{j,2,1}^2=0, & 3\leq j\leq n,\\[1mm]
\{E_{i,j},E_{j,2},H_1\}, & 3\leq i \neq j \leq n & \Rightarrow & [E_{i,2},H_1]=0, & 3\leq i \leq n,\\[1mm]
\{E_{i,j},E_{j,1},H_2\}, & 3\leq i \neq j \leq n &\Rightarrow & [E_{i,1},H_2]=0, & 3\leq i \leq n, \\[1mm]
\{E_{i,1}, E_{1,2},H_1\}, &  3\leq i\leq n & \Rightarrow & [H_i,E_{1,2}]=0, & 3\leq i \leq n,\\[1mm]
\{E_{i,1},H_i,E_{2,j}\}, & 3\leq i \neq j \leq n & \Rightarrow & [H_1,E_{2,j}]=A_{1,2,i}^iX_j, & 3\leq i \neq j \leq n, \\[1mm]

\end{array}$$

Consider

\begin{equation}\label{e(n)eq5}
\{E_{1,2}, H_{i}, E_{i,j}\}, \ 3\leq i \neq j \leq n \Rightarrow [E_{1,2},H_j]=A_{1,2,i}^iX_j-A_{1,2,i}^jX_i, \ \ 3\leq i \neq j \leq n.
\end{equation}

Using \eqref{e(n)eq5} in the Leibniz identities for the triples
$$\{E_{1,2},H_j,E_{i,l}\}, \quad \{E_{1,3},E_{3,2},H_i\}, \ 3\leq i,j \leq n,$$
we derive
$$A_{1,2,i}^i=A_{1,2,i}^j=0, \ 3\leq i,j \leq n.$$

%

Applying Leibniz identity for the triples
$$\{E_{i,2},H_j,E_{2,j}\}, \quad \{E_{i,1},H_i, E_{1,j}\}, \quad \{E_{k,i},H_k,E_{i,2}\}, \quad \{E_{k,i},H_k,E_{i,1}\}$$ and analyzing the obtained relations we deduce $[H_{i},E_{i,j}]=H_j, \ 1\leq i\neq j \leq n.$

Now we  prove the nullity of the rest products, namely, $[H_i,H_j]=0.$

From Leibniz identities we have
$$\begin{array}{llll}
\{E_{i,k}, H_i, H_{j}\}, \ 1\leq i\neq j \leq n & \Rightarrow & [H_i,H_j]=0, & 1\leq i\neq j \leq n, \\[1mm]
\{H_i, E_{i,j}, H_{j}\}, \ 1\leq i \neq j\leq n &\Rightarrow & [H_1,H_1]=[H_i,H_i], & 1\leq i \leq n,\\[1mm]
\{E_{1,i}, H_1, H_{i}\}, \ 1\leq i \leq n & \Rightarrow & [H_1,H_1]=-[H_i,H_i], & 1\leq i \leq n.
\end{array}$$

Thus, we obtain $[H_i,H_j]=0, \ \ 1\leq i, j \leq n,$ which complete the proof of theorem.
\end{proof}

\subsection{Leibniz algebras associated with Fock module over Diamond Lie algebra $\mathcal{D}_k$.}

\

Our goal in this subsection consists of extending the notion of Fock module over algebra $\mathcal{D}_k$ and to clarify the structure of Leibniz algebra associated with Fock module over algebra $\mathcal{D}_k$.

First, we introduce notations for the basis $\{X_i,Y_i,T,Z, \ \  i=1, 2, \dots, k\}$ of the algebra $\mathcal{D}_k:$
$$\overline{e}=T, \ \ \overline{1}=Z, \ \ \overline{x}_i=X_i, \ \ \frac{\overline{\delta}}{\delta x_i}=Y_i, \ \ i=1, 2, \dots, k.$$

We define the Fock module over Diamond Lie algebra $\mathcal{D}_k$ as a vector space $\mathbb{C}[x_1,\dots,x_k]$ with the action $(-,-) \ : \  \mathcal{D}_k \times \mathbb{C}[x_1,\dots,x_k] \rightarrow \mathbb{C}[x_1,\dots,x_k]$ in the following way:

\begin{equation}\label{eqfock}
\begin{array}{lll}
(p(x_1,\dots,x_k),\overline{1})&\mapsto
&p(x_1,\dots,x_k),\\[1mm]
(p(x_1,\dots,x_k),\overline{x}_i)&\mapsto
&x_ip(x_1,\dots,x_k),\\[1mm]
(p(x_1,\dots,x_k),\frac{\overline{\delta}}{\delta x_i})&\mapsto
&\frac{\overline{\delta}}{\delta
x_i}(p(x_1,\dots,x_k)),\\[1mm]
(p(x_1,\dots,x_k),\overline{e})&\mapsto
&-x_i\frac{\overline{\delta}}{\delta
x_i}(p(x_1,\dots,x_k)),\end{array}
\end{equation}
for any $p(x_1,\dots,x_k)\in \mathbb{C}[x_1,\dots,x_k]$ and $i=1,\dots,k.$

It is easy to check that this action satisfies the right module structure over algebra $\mathcal{D}_k$ and it is induced from Fock right module over Heisenberg Lie algebra.

\begin{teo}\label{thm1} Any Leibniz algebra $L$ such that $L/I \cong \mathcal{D}_k$ and the ideal $I$ is being a right Fock module $\mathbb{C}[x_1,\dots,x_k]$ over $\mathcal{D}_k$ is isomorphic to the following algebra:

$$\begin{array}{lll}
[x_1^{t_1}x_2^{t_2}\dots x_k^{t_k},\overline{1}]&=&x_1^{t_1}x_2^{t_2}\dots x_k^{t_k},\\[1mm]
[x_1^{t_1}x_2^{t_2}\dots x_k^{t_k},\overline{x}_i]&=&x_1^{t_1}\dots x_{i-1}^{t_{i-1}} x_{i}^{t_{i}+1} x_{i+1}^{t_{i+1}} \dots x_k^{t_k},\\[1mm]
[x_1^{t_1}x_2^{t_2}\dots x_k^{t_k},\frac{\overline{\delta}}{\delta
x_i}]&=& t_ix_1^{t_1}\dots x_{i-1}^{t_{i-1}} x_{i}^{t_{i}-1}x_{i+1}^{t_{i+1}}\dots x_k^{t_k},\\[1mm]
[x_1^{t_1} x_2^{t_2} \dots x_k^{t_k}, \overline{e}]&=
&-t_ix_1^{t_1} \dots x_{i-1}^{t_{i-1}} x_{i}^{t_{i}} x_{i+1}^{t_{i+1}} \dots x_k^{t_k},\end{array}$$
$$\begin{array}{lll}
[\overline{x}_i,\frac{\overline{\delta}}{\delta x_i}]=-[\frac{\overline{\delta}}{\delta x_i},\overline{x}_i]=\overline{1}, & \\[1mm]
[\overline{e},\overline{x}_i]=-[\overline{x}_i,\overline{e}]=\frac{\overline{\delta}}{\delta x_i},&\\[1mm]
[\overline{e},\frac{\overline{\delta}}{\delta x_i}]=-[\frac{\overline{\delta}}{\delta x_i},\overline{e}]=-\overline{x}_i,& \end{array}$$
for $i=1,\dots,k.$
\end{teo}
\begin{proof}

Let $L$ be a Leibniz algebra satisfying the condition of theorem. As a basis of the algebra $L$ choose $$\{\overline{e}, \ \overline{1}, \ \overline{x}_i,\ \frac{\overline{\delta}}{\delta x_i}, \ x_1^{t_1}x_2^{t_2}\cdots x_k^{t_k} |\ t_{i}\in \mathbb{N}\cup \{0\}, 1\leq i\leq k\}.$$

Set for $p_i, q_i, r, m\in \mathbb{C}[x_1, \dots, x_k]$ the following
$$\begin{array}{ll}
[\overline{x}_i,\overline{1}]=
p_i(x_1,\dots,x_k),&1\leq i\leq k,\\[1mm]
[\frac{\overline{\delta}}{\delta x_i},\overline{1}]=
q_i(x_1,\dots,x_k),&1\leq i\leq k,\\[1mm]
[\overline{1},\overline{1}]= r(x_1,\dots,x_k),&\\[1mm]
[\overline{e},\overline{1}]= m(x_1,\dots,x_k).&\end{array}$$

Taking the change of basis elements
$$\begin{array}{ll}
\overline{x}_i^\prime=\overline{x}_i-
p_i(x_1,\dots,x_k),&1\leq i\leq k,\\[1mm]
\frac{\overline{\delta}}{\delta
x_i}^\prime=\frac{\overline{\delta}}{\delta x_i}-q_i(x_1,\dots,x_k),&1\leq i\leq k,\\[1mm]
\overline{1}^\prime=\overline{1}- r(x_1,\dots,x_k),&\\[1mm]
\overline{e}^\prime=\overline{e}-m(x_1,\dots,x_k),&\end{array}$$
apply the products generated from \eqref{eqfock}. One can assume
\begin{equation}\label{center}
[\overline{x}_i,\overline{1}]=[\frac{\overline{\delta}}{\delta x_i},\overline{1}]=
[\overline{1},\overline{1}]= [\overline{e},\overline{1}]=0, \ \ \
1\leq i\leq k.
\end{equation}
The products \eqref{center} imply $\overline{1}\in Ann_r(\mathcal{D}_k)$.

From the chain of equalities
$$[[\mathcal{D}_k,\mathcal{D}_k],\overline{1}]=[\mathcal{D}_k,[\mathcal{D}_k,\overline{1}]]+[[\mathcal{D}_k,
\overline{1}],\mathcal{D}_k]=0$$
and the products $[x_1^{t_1}x_2^{t_2}\dots x_k^{t_k},\overline{1}]=x_1^{t_1}x_2^{t_2}\dots x_k^{t_k}$
we conclude
$$\begin{array}{lll}
[\overline{x}_i,\frac{\overline{\delta}}{\delta x_i}]=-[\frac{\overline{\delta}}{\delta x_i},\overline{x}_i]=\overline{1}, & 1\leq i\leq k,\\[1mm]
[\overline{e},\overline{x}_i]=-[\overline{x}_i,\overline{e}]=\frac{\overline{\delta}}{\delta x_i},& 1\leq
i\leq k,\\[1mm]
[\overline{e},\frac{\overline{\delta}}{\delta x_i}]=-[\frac{\overline{\delta}}{\delta x_i},\overline{e}]=-\overline{x}_i,& 1\leq i\leq k.
\end{array}$$
\end{proof}

\section{Acknowledgments}

This work was supported by Ministerio de Econom\'ia y Competitividad (Spain) grant MTM2013-43687-P (European FEDER support included) and by Ministry of Education and Science of the Republic of Kazakhstan the grant No. 0828/GF4.

\end{document}